\documentclass[10pt,reqno]{amsart}

\usepackage{amsmath,amsfonts,amsbsy,amsgen,amscd,mathrsfs,amssymb,amsthm}
\usepackage{enumerate}

\usepackage[usenames,dvipsnames]{xcolor}
\usepackage[colorlinks=true,citecolor=blue,linkcolor=blue]{hyperref}

\newtheorem{thm}{Theorem}[section]
\newtheorem{lem}[thm]{Lemma}

\newtheorem{cor}[thm]{Corollary}

\theoremstyle{definition}

\newtheorem{rem}[thm]{Remark}

\numberwithin{equation}{section} 
\numberwithin{figure}{section}
\numberwithin{table}{section}

\newcommand{\uh}{\mathrm{h}}

\newcommand{\Vol}{\mathrm{Vol}}

\newcommand{\diag}{\mathop{\mathrm{diag}}}
\newcommand{\D}{\mathsf{D}}
\newcommand{\V}{\mathsf{V}}
\renewcommand{\div}{\mathrm{div}}

\begin{document}

\title{Mixed volumes and the Bochner Method}
\thanks{This work was supported in part by NSF grants CAREER-DMS-1148711
and DMS-1811735, ARO through PECASE award W911NF-14-1-0094,
and the Simons Collaboration on Algorithms \& Geometry.
This work was initiated while the authors were in residence at 
MSRI in Berkeley, CA, supported by NSF grant DMS-1440140.
The hospitality of MSRI and of the organizers of the program on Geometric 
Functional Analysis is gratefully acknowledged.}

\author{Yair Shenfeld}
\address{Sherrerd Hall 323, Princeton University, Princeton, NJ
08544, USA}
\email{yairs@princeton.edu}

\author{Ramon van Handel}
\address{Fine Hall 207, Princeton University, Princeton, NJ 
08544, USA}
\email{rvan@princeton.edu}

\begin{abstract}
At the heart of convex geometry lies the observation that the volume of 
convex bodies behaves as a polynomial. Many geometric inequalities may be 
expressed in terms of the coefficients of this polynomial, called mixed 
volumes. Among the deepest results of this theory is the 
Alexandrov-Fenchel inequality, which subsumes many known inequalities as 
special cases. The aim of this note is to give new proofs of the 
Alexandrov-Fenchel inequality and of its matrix counterpart, Alexandrov's 
inequality for mixed discriminants, that appear conceptually and 
technically simpler than earlier proofs and clarify the underlying 
structure. Our main observation is that these inequalities can be reduced 
by the spectral theorem to certain trivial ``Bochner formulas''.
\end{abstract}



\maketitle

\thispagestyle{empty}

\section{Introduction and main ideas}
\label{sec:intro}

Much of the foundation for the modern theory of convex geometry was put 
forward by 
H.~Minkowski around the turn of the 20th century. One of the central 
notions in Minkowski's theory arises from the fundamental fact that the 
volume of convex bodies in $\mathbb{R}^n$ behaves as a homogeneous 
polynomial of degree $n$: that is, for any convex bodies 
$K_1,\ldots,K_m\subset\mathbb{R}^n$ and $\lambda_1,\ldots,\lambda_m>0$, we 
have
\begin{equation}
\label{eq:defmv}
	\Vol(\lambda_1K_1+\cdots+
	\lambda_mK_m) =
	\sum_{i_1,\ldots,i_n=1}^m
	\V(K_{i_1},\ldots,K_{i_n})\,
	\lambda_{i_1}\cdots\lambda_{i_n}.
\end{equation}
The coefficients $\V(K_{i_1},\ldots,K_{i_n})$ of this polynomial are 
called \emph{mixed volumes}. Given this observation, it seems 
natural to expect that many geometric properties of convex bodies may be 
expressed in terms of relations between mixed volumes. This viewpoint 
plays a major role in Minkowski's work on convex geometry \cite{Min11}, 
and lies at the heart of what is now called the Brunn-Minkowski theory 
\cite{BF87,Sch14}. Among the deepest results of this theory is the 
Alexandrov-Fenchel inequality, which subsumes many geometric inequalities 
as special cases.

\begin{thm}[Alexandrov-Fenchel inequality]
\label{thm:af}
We have
$$
	\V(K,L,C_1,\ldots,C_{n-2})^2 \ge
	\V(K,K,C_1,\ldots,C_{n-2})\,\V(L,L,C_1,\ldots,C_{n-2})
$$
for any convex bodies $K,L,C_1,\ldots,C_{n-2}$ in $\mathbb{R}^n$.
\end{thm}

The cases $n=2,3$ are special in that they can be derived from the 
Brunn-Minkowski inequality, as was already shown by Minkowski himself 
\cite[p.~261]{Min11}. However, this approach only yields special cases of 
Theorem \ref{thm:af} in higher dimension. A (questionable) proof of 
Theorem \ref{thm:af} was announced, but never published, by W.~Fenchel 
\cite{Fen36}. Finally, two different but closely related proofs were 
obtained by A.\,D.~Alexandrov \cite{Ale37,Ale38} using a homotopy method 
due to Hilbert \cite{Hil12}. It was realized much later that Theorem 
\ref{thm:af} has connections with algebraic geometry through the Hodge 
index theorem, which led to the development of algebraic and complex 
geometric proofs \cite{BZ88,Gro90,Wan18}. Despite these diverse 
viewpoints, the inequality and its proofs are generally considered to be 
conceptually deep. We refer to \cite{Sch14,AGM15} for further remarks on 
the history and significance of Theorem \ref{thm:af}.

The aim of this note is to give a new proof of the Alexandrov-Fenchel 
inequality that appears to be conceptually and technically simpler than 
previous proofs. The basic ingredients of our proof were already 
introduced by Minkowski, Hilbert, and Alexandrov. However, by means of a 
very simple but apparently overlooked device, we will replace the main 
part of Alexandrov's proof by a one-line computation. We believe the 
resulting approach is particularly intuitive and sheds new light on why 
the inequality holds. In the remainder of the introduction we describe the 
basic elements of our proof; the details are filled in in subsequent 
sections.

\subsection{Mixed volumes and mixed discriminants}
\label{sec:1mix}

Mixed volumes are defined by considering the volume of the 
sum $K+L:=\{x+y:x\in K,y\in L\}$ of convex bodies. We would like to 
think of volume as a polynomial on the space of convex bodies. However, 
this is somewhat awkward, as convex bodies do not form a vector 
space. To address this issue, we identify each convex body $K$ with its 
\emph{support function}
$$
	h_K(x) := \sup_{y\in K}\langle y,x\rangle.
$$
Geometrically, $h_K(x)$ is the distance to the origin of the supporting 
hyperplane of $K$ whose normal direction is $x\in S^{n-1}$. As $K$ can be 
recovered by intersecting all its supporting halfspaces, $h_K$ and $K$ 
uniquely determine each other.

The advantage of working with support functions is that they map 
set addition into scalar addition: $h_{aK+bL}=ah_K+bh_L$.
To understand the behavior of volume under 
addition, it is therefore natural to express
$\Vol(K)$ in terms of $h_K$: we have
\begin{equation}
\label{eq:volrep}
	\Vol(K) = \frac{1}{n}\int_{S^{n-1}}
	h_K\det(D^2h_K)\,d\omega,
\end{equation}
where $\omega$ denotes the surface measure on $S^{n-1}$ and $D^2h_K(x)$ 
denotes the restriction of the Hessian of $h_K:\mathbb{R}^n\to 
\mathbb{R}$ to the tangent space of $S^{n-1}$ at the point $x$
(this classical computation is recalled in section \ref{sec:2.1}). With 
this representation in hand, it is immediately clear that volume is a 
polynomial in the sense of \eqref{eq:defmv}: the integrand in 
\eqref{eq:volrep} is a polynomial of degree $n$ in $h_K$ in the usual 
sense (as $D^2h_K$ is an $(n-1)$-dimensional matrix), and the conclusion 
follows directly.

\begin{rem}
\label{rem:c2}
As written, the representation \eqref{eq:volrep} only makes sense for 
smooth convex bodies, that is, when $h_K$ is a $C^2$ function on 
$S^{n-1}$. However, any convex body can be approximated by smooth bodies, 
and mixed volumes are continuous with respect to this approximation 
\cite[\S 27--\S 29]{BF87}. We therefore can and will assume in the sequel 
that all convex bodies are sufficiently smooth.
\end{rem}

We can similarly represent mixed volumes in terms of support functions. 
As mixed volumes are defined as the coefficients of the polynomial 
\eqref{eq:defmv}, we must first define the analogous coefficients of 
the determinant: that is, for any $(n-1)$-dimensional matrices
$M_1,\ldots,M_m$ and $\lambda_1,\ldots,\lambda_m>0$, we define 
\begin{equation}
\label{eq:defmd}
	\det(\lambda_1M_1+\cdots+
	\lambda_mM_m) =
	\sum_{i_1,\ldots,i_{n-1}=1}^m
	\D(M_{i_1},\ldots,M_{i_{n-1}})\,
	\lambda_{i_1}\cdots\lambda_{i_{n-1}}.
\end{equation}
The coefficients $\D(M_{i_1},\ldots,M_{i_{n-1}})$ are
called \emph{mixed discriminants}. Following a similar argument to
the proof of \eqref{eq:volrep}, we obtain the following representation:
\begin{equation}
\label{eq:mvrep}
	\V(K_1,\ldots,K_n) = \frac{1}{n}
	\int_{S^{n-1}} h_{K_1}\D(D^2h_{K_2},\ldots,
	D^2h_{K_n})\,d\omega.
\end{equation}
It is important to note that mixed volumes are, by definition, symmetric
in their arguments, even though this is not obvious from
the representation \eqref{eq:mvrep}. For this reason \eqref{eq:mvrep}
does not follow trivially from \eqref{eq:volrep}. However, one can prove
\eqref{eq:mvrep} by a small modification of the proof of \eqref{eq:volrep}, 
as we will recall in section \ref{sec:2.1} below.

Now that we obtained a natural representation of mixed volumes, how might
one go about proving Theorem \ref{thm:af}? In view of 
\eqref{eq:mvrep}, one may ask first whether there is an analogue of
Theorem \ref{thm:af} for mixed discriminants. This is indeed the case.

\begin{thm}[Alexandrov's mixed discriminant inequality] 
\label{thm:a}
Let $A$ be any $(n-1)$-dimensional symmetric matrix, and let
$B,M_1,\ldots,M_{n-3}$ be $(n-1)$-dimensional
positive semidefinite matrices. Then we have
$$
	\D(A,B,M_1,\ldots,M_{n-3})^2 \ge
	\D(A,A,M_1,\ldots,M_{n-3})\,\D(B,B,M_1,\ldots,M_{n-3}).
$$
\end{thm}

Theorem \ref{thm:a} is a matrix inequality and does not necessarily belong 
to convex geometry.
Given this inequality, it might seem that the Alexandrov-Fenchel 
inequality should be a simple consequence of Theorem \ref{thm:a} and the 
representation \eqref{eq:mvrep}. This is far from clear, however. Had the 
inequality signs in Theorems \ref{thm:af} and \ref{thm:a} been reversed, 
then the former would follow directly from the latter by the 
Cauchy-Schwarz inequality. However, the inequalities being such as they 
are, Cauchy-Schwarz goes in the wrong direction and there is no reason to 
expect, \emph{a priori}, that Theorem \ref{thm:a} should imply Theorem 
\ref{thm:af}.

Theorem \ref{thm:a} was in fact used by Alexandrov in one part of his 
study of the Alexandrov-Fenchel inequality. However, in this 
proof Theorem \ref{thm:a} is used very indirectly, and the relationship 
between Theorems \ref{thm:af} and \ref{thm:a} has remained somewhat 
mysterious. Indeed, many other inequalities are known for mixed 
discriminants, but most such inequalities are simply false in the context 
of mixed volumes (e.g., \cite{AFO14}).

The new observation of this note is that when viewed in the right way, 
the Alexandrov-Fenchel inequality will prove to be a \emph{direct} 
consequence of Alexandrov's inequality for mixed discriminants. This not 
only yields a simpler proof, but also demystifies the relationship 
between Theorems \ref{thm:af} and \ref{thm:a}. We believe this 
conceptual simplification significantly clarifies the structure of 
these inequalities. Once the basic idea has been understood, we will 
find that the same idea can be used to give a simple new proof of 
Theorem \ref{thm:a}.

\subsection{Hyperbolic inequalities}
\label{sec:1hyper}

Before we can explain the main idea of this note, we must recall the 
basic structure behind the Alexandrov-Fenchel 
inequalities. By definition, mixed volumes and mixed discriminants are 
symmetric multilinear functions of their arguments. Therefore, Theorems 
\ref{thm:af} and \ref{thm:a} may be viewed as statements about certain 
\emph{quadratic forms}: Theorem \ref{thm:af} is concerned with the 
quadratic form $(h_K,h_L)\mapsto \V(K,L,C_1,\ldots,C_{n-2})$, while 
Theorem \ref{thm:a} is concerned with the quadratic form $(A,B)\mapsto 
\D(A,B,M_1,\ldots,M_{n-3})$. From this perspective, both Theorems 
\ref{thm:af} and \ref{thm:a} can be interpreted as stating that the 
relevant quadratic form satisfies a \emph{reverse} form of the 
Cauchy-Schwarz inequality.

It is instructive to recall more generally when quadratic forms satisfy 
Cauchy-Schwarz inequalities. For example, it is a basic fact of linear 
algebra that a symmetric quadratic form $\langle x,Ax\rangle$ on 
$\mathbb{R}^d$ satisfies the Cauchy-Schwarz inequality $\langle 
x,Ay\rangle^2\le \langle x,Ax\rangle\langle y,Ay\rangle$ if and only if 
the matrix $A$ is positive or negative semidefinite. The validity of the 
reverse Cauchy-Schwarz inequality can be characterized in an entirely 
analogous manner, see section \ref{sec:2.3} for a short proof.

\begin{lem}[Hyperbolic quadratic forms]
\label{lem:hyper}
Let $A$ be a symmetric matrix.
Then the following conditions are equivalent:
\begin{enumerate}[1.]
\item $\langle x,Ay\rangle^2\ge\langle x,Ax\rangle\langle y,Ay\rangle$
for all $x,y$ such that $\langle y,Ay\rangle \ge 0$.
\item The positive eigenspace of $A$ has dimension at most one.
\end{enumerate}
The conclusion remains valid if $A$ is a 
self-adjoint operator on a Hilbert space with a discrete spectrum, 
provided the vectors $x,y$ are chosen in the domain of $A$.
\end{lem}

To apply Lemma \ref{lem:hyper} to the Alexandrov-Fenchel inequality,
we may reason as follows. Fix bodies $C_1,\ldots,C_{n-2}$, and define
\begin{equation}
\label{eq:atilde}
	\mathscr{\tilde A}f := 
	\frac{1}{n}\D(D^2f,D^2h_{C_1},\ldots,D^2h_{C_{n-2}}).
\end{equation}
Then the representation \eqref{eq:mvrep} can be expressed as
$$
	\V(K,L,C_1,\ldots,C_{n-2}) =
	\langle h_K,\mathscr{\tilde A}h_L\rangle_{L^2(\omega)}.
$$ 
Note that $\mathscr{\tilde A}$ is a second-order differential operator on 
$S^{n-1}$. It will follow from basic properties of mixed discriminants and 
mixed volumes that $\mathscr{\tilde A}$ is elliptic and 
symmetric on $L^2(\omega)$. Thus standard elliptic regularity theory shows 
that $\mathscr{\tilde A}$ is self-adjoint and that it has a 
discrete spectrum and a simple top eigenvalue (cf.\ section~\ref{sec:3}). 
Therefore, by Lemma \ref{lem:hyper}, the Alexandrov-Fenchel inequality is 
\emph{equivalent} to the statement that $\mathscr{\tilde A}$ has exactly 
one positive eigenvalue.

\subsection{The Bochner method}
\label{sec:1bochner}

Up to this point we have not formally made any progress towards proving 
the Alexandrov-Fenchel inequality: we have merely reformulated the 
statement of Theorem \ref{thm:af} as an equivalent spectral problem. The 
key question in the proof of Theorem \ref{thm:af} is why the relevant 
spectral property actually holds. What is new in this note is the 
realization that this follows almost immediately from Theorem \ref{thm:a}
by a one-line computation.

Let us sketch the relevant argument. It is convenient to normalize the 
operator $\mathscr{\tilde A}$ such that its top eigenvalue is $1$.
Let us call the normalized operator $\mathscr{A}$. As 
$\mathscr{A}f$ is defined by a mixed discriminant 
\eqref{eq:atilde}, what can be deduced from Theorem \ref{thm:a} is an 
inequality for $(\mathscr{A}f)^2$: indeed, when we choose the appropriate 
normalization, integrating both sides of Theorem \ref{thm:a} will 
immediately yield the inequality
\begin{equation}
\label{eq:lichner}
	\langle\mathscr{A}f,\mathscr{A}f\rangle \ge
	\langle f,\mathscr{A}f\rangle,
\end{equation}
where the inner product is the one associated to the
normalized operator (cf.\ section~\ref{sec:3}).
By plugging in for $f$ any eigenfunction of $\mathscr{A}$, it
follows that any eigenvalue $\lambda$ of $\mathscr{A}$ must satisfy
$\lambda^2\ge\lambda$. But as the normalization was chosen such that
$\lambda_{\rm max}=1$, this can evidently only happen if either $\lambda=1$
or $\lambda\le 0$, concluding the proof.

This very simple device sheds light on the reason why an 
inequality for mixed volumes can be deduced from an inequality for mixed 
discriminants: as our inequalities are spectral in nature, the spectral 
theorem reduces the problem of bounding the square of the quadratic form 
of an operator to that of bounding the square of the operator itself. 
Once this idea has been understood, it becomes apparent that it explains 
also other aspects of the Alexandrov-Fenchel theory. For example, 
the same principle will give a new proof of Theorem \ref{thm:a}.

While our approach has apparently been overlooked in the 
literature on the Alexandrov-Fenchel inequality,\footnote{%
	However, a recent paper of Wang \cite{Wan18} uses 
	various algebraic identities in K\"ahler geometry, including
	a Bochner-type formula, to give a complex-geometric proof of
	the Alexandrov-Fenchel inequality.
	While the connection with our elementary methods is unclear
        to us, \cite{Wan18} provided the initial inspiration to
	pursue the ideas in this paper.
}
the underlying idea is classical in Riemannian geometry: it was used by 
Lichnerowicz \cite{Lic58} to lower bound the spectral gap of the Laplacian 
on a Riemannian manifold with positive Ricci curvature. In this setting, 
the analogue of \eqref{eq:lichner} is established by means of a technique
known as the Bochner method. This analogy is not a 
coincidence: for example, in the case $C_1=\cdots=C_{n-2}=B_2$ (the 
Euclidean unit ball), it turns out that \eqref{eq:lichner} 
reduces exactly to a Bochner formula for the Laplacian on 
$S^{n-1}$, see section \ref{sec:bochner} below. We emphasize, however, 
that no Riemannian geometry will be used in our proofs.

\subsection{Organization of this paper}

The rest of this note is organized as follows. Section \ref{sec:2} recalls 
basic facts about mixed volumes and mixed discriminants. In section 
\ref{sec:3}, we prove Theorem \ref{thm:af} assuming validity of Theorem 
\ref{thm:a}. In section \ref{sec:4}, our method is adapted to prove 
Theorem \ref{thm:a} itself. In section \ref{sec:app} we sketch an 
alternative proof of Theorem \ref{thm:af} that uses polytopes instead of 
smooth bodies; while we find this approach less illuminating, it has the 
advantage of using only matrices and avoiding the use of elliptic operators.
Finally, section \ref{sec:concl} contains some concluding remarks that 
places our approach in context.

\section{Basic facts}
\label{sec:2}

The aim of this section is to recall the basic properties of mixed volumes 
and mixed discriminants that will be needed in the sequel. The material in 
this section is standard, see, e.g., \cite{BF87,Sch14}. We have 
nonetheless chosen to include (almost) full proofs, both in order to make 
our exposition accessible to non-experts and to emphasize that the facts 
recalled in this section are indeed elementary. Readers who are familiar 
with basic properties of mixed volumes and mixed discriminants are 
encouraged to skip ahead directly to section \ref{sec:3}.

\subsection{Convex bodies and support functions}
\label{sec:2.0}

A \emph{convex body} is a nonempty compact convex subset of 
$\mathbb{R}^n$. We will mostly work with bodies that are sufficiently 
smooth so that the representation formulas stated in section 
\ref{sec:intro} are valid. Let us make this requirement more precise.

As support functions are $1$-homogeneous functions on $\mathbb{R}^n$, 
let us first consider such functions more generally. First of all, a 
$1$-homogeneous function $f:\mathbb{R}^n\to\mathbb{R}$, i.e., 
$f(x)=\|x\|f(x/\|x\|)$, is clearly uniquely determined by its values 
on $S^{n-1}$. Conversely, the latter identity uniquely extends any 
function $f:S^{n-1}\to\mathbb{R}$ to a $1$-homogeneous function on 
$\mathbb{R}^n$. Now note that if $f$ is $1$-homogeneous and $C^2$, then 
$\nabla f$ is $0$-homogeneous, so that $\nabla^2f(x)x=0$. The Hessian of 
$f$ is therefore completely determined by the restriction of the linear 
map $\nabla^2f(x):\mathbb{R}^n\to\mathbb{R}^n$ to the tangent space 
$x^\perp$ of the sphere. We denote this restriction as 
$D^2f(x):x^\perp\to x^\perp$.\footnote{%
	By choosing a basis of $x^\perp$, one may express $D^2f(x)$ as an 
	$(n-1)$-dimensional matrix. However, we only use 
	determinants and mixed discriminants of such matrices which are 
	basis-independent.
}
If we begin instead with a $C^2$ function $f$ on $S^{n-1}$, then we 
denote by $D^2f(x)$ for $x\in S^{n-1}$ the restricted Hessian of its 
$1$-homogeneous extension.

The restricted Hessian $D^2f$ appears naturally when performing calculus
with support functions. For example, we have the following basic
result.\footnote{
	The notation $M>0$ ($M\ge 0$) denotes that
	$M$ is positive definite (positive semidefinite).}

\begin{lem}
\label{lem:d2supp}
Let $f:S^{n-1}\to\mathbb{R}$ be a $C^2$ function. Then $f=h_K$ for some
convex body $K$ if and only if $D^2f(x)\ge 0$ for all $x\in S^{n-1}$.
\end{lem}

\begin{proof}
As support functions are convex, clearly $D^2h_K\ge 0$. Conversely,
suppose that $D^2f\ge 0$. Then the $1$-homogeneous extension of $f$ is 
convex, so it can be written as the supremum of
affine functions $f(x)=\sup_{y\in A} \{\langle y,x\rangle -f^*(y)\}$.
It is readily verified that $1$-homogeneity implies $f^*=0$, and
that $A$ is bounded as $f$ is finite.
Thus $f(x)=\sup_{y\in A}\langle y,x\rangle=
h_{\overline{\mathrm{conv}}(A)}(x)$.
\end{proof}

An key corollary is that any $C^2$ function is a difference of 
support functions.

\begin{cor}
\label{cor:c2}
Let $f:S^{n-1}\to\mathbb{R}$ be a $C^2$ function and $L$ be a convex
body such that $D^2h_L>0$. Then there is a convex body $K$ and 
$a>0$ such that $f=a(h_K-h_{L})$. In particular, any $C^2$ function on 
$S^{n-1}$ is the difference of two support functions.
\end{cor}

\begin{proof}
As $S^{n-1}$ is compact and $f,h_L$ are $C^2$ functions, we have
$D^2f\ge -\alpha I$ and $D^2h_L\ge \beta I$
for some $\alpha,\beta>0$. Thus
$g:=f+(\alpha/\beta)h_L$ satisfies $D^2g\ge 0$, so
$f = (\alpha/\beta)(h_K - h_L)$ 
for some convex body $K$ by Lemma \ref{lem:d2supp}.
We may always choose $L=B_2$ to be
the Euclidean ball (as $D^2h_{B_2}=I$).
\end{proof}

A convex body $K$ is of class $C^k_+$ ($k\ge 2$) if its support function 
$h_K$ is $C^k$ and satisfies $D^2h_K>0$. Such bodies will allow us to 
perform all the calculus we need; see \cite[section~2.5]{Sch14} for a 
detailed study of the regularity of such bodies. For our purposes, working 
with $C^\infty_+$ bodies entails no loss of generality, cf.\ Remark 
\ref{rem:c2}. As the approximation argument is unrelated to the topic of 
this paper, we omit further discussion and refer instead to \cite[sections 
3.4 and 5.1]{Sch14}.

\subsection{Representation of volumes and mixed volumes}
\label{sec:2.1}

We now prove \eqref{eq:volrep} and \eqref{eq:mvrep}. To prove
\eqref{eq:volrep}, we first use the divergence 
theorem to write $\Vol(K)$ as an integral over $\partial K$; 
then we change variables using the outer unit normal vector 
$n_K:\partial K\to S^{n-1}$ to map the integral to $S^{n-1}$. The 
term $\det(D^2h_K)$ that appears in \eqref{eq:volrep} is precisely the
Jacobian of this transformation.

\begin{lem}
\label{lem:volrep}
Let $K$ be a $C^2_+$ convex body. Then
$$
	\Vol(K) = \frac{1}{n}\int_{S^{n-1}}h_K\det(D^2h_K)\,d\omega.
$$
\end{lem}

\begin{proof}
By the divergence theorem,
$$
	\Vol(K) = \frac{1}{n} \int_K\div(x)\,dx 
	= \int_{\partial K} \langle x,n_K(x)\rangle\,d\omega_{K}(x),
$$
where $\omega_{K}$ is the surface measure on $\partial K$
and $n_K$ is the outer unit normal.
Now note that $\nabla h_K$ (the gradient is in $\mathbb{R}^n$) 
maps $u\in S^{n-1}$ to $\nabla h_K(u)=\mathrm{arg\,max}_{y\in K}\langle 
y,u\rangle \in\partial K$. As $D^2h_K>0$, the map $\nabla 
h_K:S^{n-1}\to \partial K$ is a diffeomorphism.
Thus
$$
	\Vol(K) = \frac{1}{n} \int_{S^{n-1}}
	\langle \nabla h_K,n_K(\nabla h_K)\rangle
	\det(D^2h_K)\,d\omega
$$
by the change of variables formula.
It remains to note that $\nabla h_K=n_K^{-1}$: indeed, as
$\langle y-x,n_K(x)\rangle \le 0$ for $x\in\partial K$ and $y\in K$
by convexity, we have
$\nabla h_K(n_K(x))=\mathrm{arg\,max}_{y\in K}\langle y,n_K(x)\rangle = x$.
As clearly $\langle \nabla h_K(u),u\rangle = \max_{y\in K}
\langle y,u\rangle=h_K(u)$, it follows that
$\langle \nabla h_K,n_K(\nabla h_K)\rangle = h_K$, and the proof is 
complete.
\end{proof}

Lemma \ref{lem:volrep} shows that volume is a polynomial in the sense of 
\eqref{eq:defmv}, but this does not immediately yield \eqref{eq:mvrep}: 
choosing $K=\lambda_1K_1+\cdots+\lambda_nK_n$ in Lemma \ref{lem:volrep} 
and using \eqref{eq:defmd} would give \eqref{eq:mvrep} averaged over all 
permutations of $K_1,\ldots,K_n$. To prove a non-symmetric representation, 
it is convenient to first prove a special case.

\begin{lem}
\label{lem:volvar}
Let $K,L$ be $C^2_+$ convex bodies. Then
$$
	\V(K,L,\ldots,L) = \frac{1}{n}\int_{S^{n-1}}
	h_K\det(D^2h_L)\,d\omega.
$$
\end{lem}

\begin{proof}
The idea is to repeat the proof of Lemma \ref{lem:volrep}, but replacing 
$\div(x)$ by $\div(Y)$ for some suitably chosen vector field $Y$. More 
precisely, let $Y$ be a bounded Lipschitz vector field. Then $I-t\nabla 
Y$ is nonsingular for sufficiently small $t$. Therefore
\begin{align*}
	&
	\lim_{t\to 0}\frac{1}{t}\bigg\{
	\int_{\mathbb{R}^n} 1_L(x-tY(x))\,dx - \Vol(L)\bigg\} = \\
	&
	\lim_{t\to 0}\int_{\mathbb{R}^n}
	1_L(x-tY(x))\,\frac{1-\det(I-t\nabla Y(x))}{t}\,dx
	=
	\int_L \div(Y)\,dx = \int_{\partial L} \langle Y,n_L\rangle
	\,d\omega_L,
\end{align*}
where we used the change of variables formula in the first step, and the
divergence theorem in the last step.
Now take the supremum on both sides over Lipschitz vector fields $Y$
taking values in $K$. As $1_L(x-tY(x))\le 1_{L+tK}(x)$ for any such
$Y$,
\begin{align*}
	n\V(K,L,\ldots,L) &=
	\lim_{t\to 0}\frac{\Vol(L+tK)-\Vol(L)}{t}
	\\
	&\ge \int_{\partial L} h_K(n_L)\,d\omega_L
	=
	\int_{S^{n-1}} h_K\det(D^2h_L)\,d\omega,
\end{align*}
where we changed variables in the last step using $\nabla h_L$ as in 
Lemma \ref{lem:volrep}.

To obtain the reverse inequality, note that by Corollary \ref{cor:c2}, 
there is a $C^2_+$ body $C$ and $a>0$ such that $-h_K = a(h_C-h_L)$. As 
mixed volumes are linear in each argument (this follows from 
\eqref{eq:defmv}), $\V(K,L,\ldots,L) = a(\Vol(L) - \V(C,L,\ldots,L))$. 
Applying the above inequality to $\V(C,L,\ldots,L)$ and Lemma 
\ref{lem:volrep}, we readily obtain the reversed inequality for 
$\V(K,L,\ldots,L)$.
\end{proof}

Choosing $K=K_1$, $L=\lambda_2K_2+\cdots+\lambda_nK_n$ in Lemma 
\ref{lem:volvar}, and applying the definitions \eqref{eq:defmv} and 
\eqref{eq:defmd} of mixed volumes and discriminants, directly yields 
\eqref{eq:mvrep}.

\begin{cor}
\label{cor:mvrep}
Let $K_1,\ldots,K_n$ be $C^2_+$ convex bodies. Then
$$
	\V(K_1,\ldots,K_n) = 
	\frac{1}{n}\int_{S^{n-1}} h_{K_1}
	\D(D^2h_{K_2},\ldots,D^2h_{K_n})\,d\omega.
$$
\end{cor}

\subsection{Basic properties of mixed volumes and mixed discriminants}
\label{sec:2.2}

We now proceed to recall the basic properties of mixed volumes and mixed 
discriminants.

\begin{lem}[Properties of mixed discriminants]
\label{lem:mdprop}
Let $M,M_1,\ldots,M_{n-1}$ be symmetric $(n-1)$-dimensional matrices
and $U$ be an $(n-1)$-dimensional matrix.
\begin{enumerate}[(a)]
\item $\D(M,\ldots,M) =\det(M)$.
\item $\D(M_1,\ldots,M_{n-1})$ is symmetric and multilinear in its
arguments.
\item $\D(UM_1U^*,\ldots,UM_{n-1}U^*) =
\det(UU^*)\D(M_1,\ldots,M_{n-1})$.
\item $\D(M_1,\ldots,M_{n-1})\ge 0$ 
if $M_1,M_2,\ldots,M_{n-1}\ge 0$.
\item $\D(M_1,\ldots,M_{n-1})>0$ if $M_2,\ldots,M_{n-1}>0$
and $M_1\geq 0$, $M_1\ne 0$.
\item $\D(e_ie_i^*,M_2,\ldots,M_{n-1}) = \frac{1}{n-1}
\D(M_2^{\langle i\rangle},\ldots,M_{n-1}^{\langle i\rangle})$, where 
$\{e_i\}$ is the standard basis in $\mathbb{R}^{n-1}$ and
$M^{\langle i\rangle}$ is obtained from $M$ by removing
its $i$-th row and column.
\end{enumerate}
\end{lem}

\begin{rem}
\label{rem:mdlowdim}
Note that, by definition, the mixed discriminant of $k$-dimensional 
matrices has $k$ arguments. Therefore, as no confusion can arise, we
denote mixed discriminants in every dimension by the same symbol $\D$
(e.g., as in Lemma \ref{lem:mdprop}(\textit{f})).
\end{rem}

\begin{proof}
Parts (\textit{a}) and (\textit{b}) follow directly from the definition
\eqref{eq:defmd}. Part (\textit{c}) also follows from 
\eqref{eq:defmd} using $\det(UMU^*)=\det(UU^*)\det(M)$.
For the remaining parts, it is useful to compute the mixed discriminant of 
rank one matrices. Let $v_1,\ldots,v_{n-1}\in\mathbb{R}^{n-1}$ be the 
columns of a matrix $V$. Then
$\det\big(\sum_{i=1}^{n-1}v_iv_i^*\big)=\det(VV^*)=\det(V)^2$.
By scaling $v_i$ we obtain 
$\det\big(\sum_{i=1}^{n-1}\lambda_iv_iv_i^*\big)=\lambda_1\cdots\lambda_{n-1}\det(V)^2$,
so \eqref{eq:defmd} implies
\begin{equation}
\label{eq:rank1}
	\D(v_1v_1^*,\ldots,v_{n-1}v_{n-1}^*) =
	\frac{\det(V)^2}{(n-1)!}\ge 0.
\end{equation}
Part (\textit{d}) now follows from linearity of mixed discriminants, as 
any $M\ge 0$ can be written as the sum of rank one matrices of the form 
$vv^*$. If $M_1\ge 0$, $M_1\ne 0$ and $M_i>0$ for $i\ge 2$, we can write 
$M_i=M_i'+v_iv_i^*$ for each $i$ where $M_i'\ge 0$ and 
$v_1,\ldots,v_{n-1}$ are linearly independent. Then part (\textit{e}) 
follows by observing that 
$\D(v_1v_1^*,\ldots,v_{n-1}v_{n-1}^*)>0$
by \eqref{eq:rank1}. Finally, part (\textit{f})
follows for $M_i=v_iv_i^*$ directly from \eqref{eq:rank1}, and 
extends to general $M_i$ by linearity.
\end{proof}

\begin{lem}[Properties of mixed volumes]
\label{lem:mvprop}
Let $K,K_1,\ldots,K_n$ be convex bodies.
\begin{enumerate}[(a)]
\item $\V(K,\ldots,K)=\Vol(K)$.
\item $\V(K_1,\ldots,K_n)$ is symmetric and multilinear in its arguments.
\item $\V(K_1,\ldots,K_n)$ is invariant under translation $K_i\mapsto
K_i+z_i$.
\item $\V(K_1,\ldots,K_n)\ge 0$.
\end{enumerate}
\end{lem}

\begin{proof}
Parts (\textit{a}) and (\textit{b}) follow directly from the definition
\eqref{eq:defmv}. Part (\textit{c}) also follows from
\eqref{eq:defmv} using $\Vol(K)=\Vol(K+z)$. To prove part (\textit{d}), we 
may assume without loss of generality
that $0\in K_1$ by translation-invariance, which implies $h_{K_1}\ge 0$. 
Then part (\textit{d}) follows for $C^2_+$ bodies from
Corollary \ref{cor:mvrep} and Lemma~\ref{lem:mdprop}(\textit{d}),
and for general bodies by approximation (cf.\ Remark \ref{rem:c2}).
\end{proof}

\subsection{Hyperbolic quadratic forms}
\label{sec:2.3}

We conclude with a proof of Lemma \ref{lem:hyper}; we in fact add an 
equivalent condition that will be useful in the proof of Theorem 
\ref{thm:a}.

\begin{lem}[Hyperbolic quadratic forms]
\label{lem:hyperfull}
Let $A$ be a symmetric matrix. Then the following conditions are 
equivalent:
\begin{enumerate}[1.]
\item $\langle x,Ay\rangle^2\ge\langle x,Ax\rangle\langle y,Ay\rangle$
for all $x,y$ such that $\langle y,Ay\rangle \ge 0$.
\item There exists a vector $w$ such that
$\langle x,Ax\rangle\le 0$ for all $x$ such that
$\langle x,Aw\rangle=0$.
\item The positive eigenspace of $A$ has dimension at most one.
\end{enumerate}
The conclusion remains valid if $A$ is a self-adjoint operator
on a Hilbert space with a discrete spectrum, provided the vectors
$x,y,w$ are chosen in the domain of $A$.
\end{lem}

\begin{proof}
If $A$ is negative semidefinite, the conclusion is trivial. Let us
therefore assume that $A$ has an eigenvector $v$ with 
positive eigenvalue $\lambda>0$.

$3\Rightarrow 2$: by assumption, the second-largest eigenvalue $\lambda_2$ 
of $A$ is nonpositive, so
$$
	0\ge \lambda_2 = \max\{\langle x,Ax\rangle:
	\|x\|=1,~\langle x,v\rangle=0\}.
$$
As $\lambda \langle x,v\rangle = \langle x,Av\rangle$, we may choose $w=v$.

$2\Rightarrow 1$: assume $\langle y,Ay\rangle>0$ (else
the conclusion is trivial). Then $\langle y,Aw\rangle\ne 0$,
so we may define $z=x-ay$ with 
$a=\langle x,Aw\rangle/\langle y,Aw\rangle$. As $\langle z,Aw\rangle=0$,
we obtain
$$
	0 \ge \langle z,Az\rangle =
	\langle x,Ax\rangle - 2a\langle x,Ay\rangle + a^2\langle y,Ay\rangle
	\ge
	\langle x,Ax\rangle - \frac{\langle x,Ay\rangle^2}{\langle y,Ay\rangle},
$$
where the last inequality is obtained by minimizing over $a$.

$1\Rightarrow 3$: let $u\perp v$ be an eigenvector of $A$ 
with eigenvalue $\mu$. Then we obtain $0 = \langle v,Au\rangle^2
\ge \lambda\mu\|v\|^2\|u\|^2$. As $\lambda>0$, we must have $\mu\le 0$.
\end{proof}

\begin{rem}
The assumption that $A$ has a discrete spectrum ensures that the proof 
extends \emph{verbatim} to the infinite-dimensional setting (for the 
variational characterization of eigenvalues used in the proof of 
$3\Rightarrow 2$, see, e.g., \cite[eq.\ (8.94)]{GT01}). This assumption is 
not really necessary, see \cite[p.\ 184]{BZ88} for a more general 
formulation. However, the present simple formulation suffices for our 
purposes.
\end{rem}

\section{The Alexandrov-Fenchel inequality}
\label{sec:3}

In this section we will prove the Alexandrov-Fenchel inequality assuming 
the validity of Alexandrov's inequality for mixed discriminants. The idea 
of the proof was already explained in section \ref{sec:1bochner}, and it 
remains to spell out the details. 

Throughout this section, we fix $C^\infty_+$ convex bodies 
$C_1,\ldots,C_{n-2}$. For reasons that will become clear shortly, we will 
also assume that $0\in\mathop{\mathrm{int}}C_1$. The latter entails no 
loss of generality: $C^\infty_+$ bodies have nonempty interior, and thus 
we may assume $0\in\mathop{\mathrm{int}}C_1$ by translation-invariance of 
mixed volumes (Lemma \ref{lem:mvprop}(\textit{c})).

We begin by expressing mixed volume as the quadratic form of a suitably 
chosen operator. While the most obvious choice is \eqref{eq:atilde}, we do 
not know much \emph{a priori} about where its eigenvalues are located.
Instead, we will choose a different normalization that fixes the top
eigenvalue. To this end, let us define 
$$
	\mathscr{A}f := 
	\frac{h_{C_1}\D(D^2f,D^2h_{C_1},\ldots,D^2h_{C_{n-2}})}{
	\D(D^2h_{C_1},D^2h_{C_1},\ldots,D^2h_{C_{n-2}})}
$$
for any $C^2$ function $f$.
That is, $\mathscr{A}f$ is obtained by rescaling the operator of
\eqref{eq:atilde} by some positive function. Correspondingly, if we 
define a measure on $S^{n-1}$ by
$$
	d\mu := \frac{1}{n}
	\frac{\D(D^2h_{C_1},D^2h_{C_1},\ldots,D^2h_{C_{n-2}})}
	{h_{C_1}}\,d\omega,
$$
then \eqref{eq:mvrep} can clearly be written as
$$
	\V(K,L,C_1,\ldots,C_{n-2}) =
	\langle h_K,\mathscr{A}h_L\rangle_{L^2(\mu)}
	:= \int h_K\,\mathscr{A}h_L\,d\mu.
$$
Note that all the above objects are well defined, as $h_{C_1}>0$ because 
we assumed $0\in\mathop{\mathrm{int}}C_1$,
and as $\D(D^2h_{C_1},D^2h_{C_1},\ldots,D^2h_{C_{n-2}})>0$ by
Lemma \ref{lem:mdprop}(\textit{e}).

The point of scaling the operator in this manner is that now, by 
definition, $\mathscr{A}h_{C_1}=h_{C_1}$. Thus $\mathscr{A}$ has 
eigenvalue $1$, and an associated eigenvector $h_{C_1}$ that is strictly 
positive. Let us collect a few basic facts about the operator $\mathscr{A}$.
\begin{enumerate}[$\bullet$]
\itemsep\abovedisplayskip
\item $\mathscr{A}$ is a uniformly elliptic operator (it is increasing 
as a function of $D^2f$ in the positive semidefinite order); this follows 
from Lemma \ref{lem:mdprop}(\textit{e}).
\item $\mathscr{A}$ defines a symmetric quadratic form 
$\langle f,\mathscr{A}g\rangle_{L^2(\mu)}=
\langle g,\mathscr{A}f\rangle_{L^2(\mu)}$
for $f,g\in C^2$; this follows from
Lemma \ref{lem:mvprop}(\textit{b}) and Corollary \ref{cor:c2}.
\item $\mathscr{A}$ extends to a self-adjoint operator 
with a discrete spectrum; its largest eigenvalue is $1$ and the
corresponding eigenspace is spanned by $h_{C_1}$; and all its 
eigenfunctions are $C^\infty$. This follows from standard elliptic 
regularity theory \cite[\S 8.12]{GT01}.
\end{enumerate}
These facts may be viewed in essence as an infinite-dimensional analogue 
of the Perron-Frobenius theorem \cite{BR97}: a uniformly elliptic operator 
on a compact manifold behaves much like a positive matrix, in 
particular, it has a unique positive eigenvector and the associated 
eigenvalue is maximal. The use of elliptic operators is convenient but not 
essential; an alternative approach is sketched in section 
\ref{sec:app}.

We now arrive at the key observation of this paper.

\begin{lem}
\label{lem:lich}
For any function $f\in C^2$, we have
$$
	\langle \mathscr{A}f,\mathscr{A}f\rangle_{L^2(\mu)}
	\ge
	\langle f,\mathscr{A}f\rangle_{L^2(\mu)}.
$$
\end{lem}

\begin{proof}
In the present notation, the statement of Theorem \ref{thm:a} can be 
written as
$$
	(\mathscr{A}f)^2 
	\ge
	h_{C_1}^2
	\frac{\D(D^2f,D^2f,D^2h_{C_2},\ldots,D^2h_{C_{n-2}})}{
	\D(D^2h_{C_1},D^2h_{C_1},\ldots,D^2h_{C_{n-2}})}.
$$
Integrating both sides with respect to $\mu$ yields
\begin{align*}
	\int (\mathscr{A}f)^2\,d\mu &\ge
	\frac{1}{n}
	\int h_{C_1}\D(D^2f,D^2f,D^2h_{C_2},\ldots,D^2h_{C_{n-2}})\,d\omega 
	\\ &= 
	\frac{1}{n}
	\int f\,\D(D^2f,D^2h_{C_1},\ldots,D^2h_{C_{n-2}})\,d\omega =
	\langle f,\mathscr{A}f\rangle_{L^2(\mu)},
\end{align*}
where we used the symmetry of mixed volumes to exchange the role
of $h_{C_1}$ and $f$ (using Corollary \ref{cor:mvrep}, 
Lemma \ref{lem:mvprop}(\textit{b}), and Corollary \ref{cor:c2}).
\end{proof}

The proof of the Alexandrov-Fenchel inequality is now almost immediate.

\begin{proof}[Proof of Theorem \ref{thm:af}]
Let $f$ be an eigenfunction of $\mathscr{A}$ with eigenvalue $\lambda$.
Then Lemma \ref{lem:lich} yields
$\lambda^2\ge\lambda$, so $\lambda\ge 1$ or $\lambda\le 0$. 
Thus the positive eigenspace of $\mathscr{A}$ is spanned by
$h_{C_1}$, and we conclude by invoking Lemma \ref{lem:hyper}.
\end{proof}

\begin{rem}
\label{rem:generalf}
The proof of Theorem \ref{thm:af} shows that $\mathscr{A}$ has a 
one-dimensional positive eigenspace, so the Alexandrov-Fenchel inequality 
follows from Lemma \ref{lem:hyper}. While we did not use this in the 
proof, we stated in the introduction that the Alexandrov-Fenchel 
inequality is in fact \emph{equivalent} to this spectral statement.
This may not be entirely obvious, however, as the Alexandrov-Fenchel
inequality only yields condition 1 of Lemma \ref{lem:hyper}
when $x,y$ are support functions.

For completeness, let us show that the spectral property of $\mathscr{A}$ 
is in fact also a consequence of the Alexandrov-Fenchel 
inequality. Let $f$ be any $C^2$ function. By Corollary \ref{cor:c2}, 
$f+ah_{C_1}$ is a support function for $a$ sufficiently large, so that
$$
	\langle f+ah_{C_1},\mathscr{A}h_{C_1}\rangle_{L^2(\mu)}^2
	\ge
	\langle f+ah_{C_1},\mathscr{A}(f+ah_{C_1})\rangle_{L^2(\mu)}
	\langle h_{C_1},\mathscr{A}h_{C_1}\rangle_{L^2(\mu)}
$$
by the Alexandrov-Fenchel inequality.
Expanding both sides yields
$$
	\langle f,\mathscr{A}h_{C_1}\rangle_{L^2(\mu)}^2 \ge
	\langle f,\mathscr{A}f\rangle_{L^2(\mu)}
	\langle h_{C_1},\mathscr{A}h_{C_1}\rangle_{L^2(\mu)}.
$$
If we now choose $f\perp h_{C_1}$ to be any eigenfunction of $\mathscr{A}$ 
with eigenvalue $\mu$, this inequality shows that $\mu\le 0$, 
establishing the claim.
\end{rem}

\section{Alexandrov's mixed discriminant inequality}
\label{sec:4}

In this section we will prove Theorem \ref{thm:a} using the same method as 
in section \ref{sec:3}. The main new difficulty is that the mixed 
discriminant inequality is an inequality for matrices rather than for 
vectors: as matrix multiplication is noncommutative, it is not clear how 
to define the normalized operator as in the previous section. It turns out 
that a second application of Lemma \ref{lem:hyperfull} allows us to reduce 
the problem to a special case where the relevant matrices are diagonal; 
the latter can be handled by repeating almost verbatim the argument of 
section \ref{sec:3}.

In the present setting, the proof proceeds by induction on the dimension.
Let us first dispose of the base of the induction, which follows from a 
trivial computation.

\begin{lem}
\label{lem:2dmd}
Let $A,B$ be $2\times 2$ matrices. Then $\D(A,B)^2\ge\D(A,A)\,\D(B,B)$.
\end{lem}

\begin{proof}
The general case is reduced to the case $B=I$ by applying Lemma 
\ref{lem:mdprop}(\textit{c}) with $U=B^{-1/2}$ to both
sides of the inequality. Moreover, by an appropriate choice of basis,
we may assume without loss of generality that $A$ is diagonal.
Then we have $\det(A+tI)=(a_{11}+t)(a_{22}+t)$, so
$\D(A,I)=\tfrac{1}{2}(a_{11}+a_{22})$ and $\D(A,A)=a_{11}a_{22}$.
Thus the desired inequality $(a_{11}+a_{22})^2\ge 4a_{11}a_{22}$
is elementary.
\end{proof}

We now proceed with the induction argument: in the remainder of this 
section we assume that Theorem \ref{thm:a} is valid for 
$(n-1)$-dimensional matrices (for $n\ge 3$), and we will show that it must 
also be valid for $n$-dimensional matrices.

We begin by proving a ``commutative'' special case: note that the 
quadratic form in the following proof acts on vectors rather than 
matrices.

\begin{lem}
\label{lem:mdcommut}
Let $n\ge 3$ and let $M_2,\ldots,M_{n-2}$ be $n$-dimensional positive 
definite matrices.
Then for any $n$-dimensional \emph{diagonal} matrix $Z$, we have
$$
	\D(Z,I,I,M_2,\ldots,M_{n-2})^2
	\ge 
	\D(Z,Z,I,M_2,\ldots,M_{n-2})\,
	\D(I,I,I,M_2,\ldots,M_{n-2}).
$$
\upn{(}When $n=3$, the statement should be read as 
$\D(Z,I,I)^2\ge\D(Z,Z,I)\,\D(I,I,I)$.\upn{)}
\end{lem}

\begin{proof}
Define for $x,y\in\mathbb{R}^n$ the quadratic form
\begin{align*}
	Q(x,y) &:=
	\D(\diag(x),\diag(y),I,M_2,\ldots,M_{n-2}) \\
	&\phantom{:}=
	\frac{1}{n}
	\sum_{i=1}^{n} x_i\,
	\D(\diag(y)^{\langle i\rangle},I^{\langle i\rangle},
	M_2^{\langle i\rangle},\ldots,M_{n-2}^{\langle i\rangle}),
\end{align*}
where we used Lemma 
\ref{lem:mdprop}(\textit{f}) (recall that
$M^{\langle i\rangle}$ is the $(n-1)$-dimensional matrix
obtained from the $n$-dimensional matrix $M$ by
removing its $i$th row and column). This formula will play the 
role of \eqref{eq:mvrep} in the present setting.

We now proceed as in section \ref{sec:3}.
Define the $n\times n$ matrix $A$ and $p\in\mathbb{R}^n$ by
\begin{align*}
	(Ay)_i &:= 
	\frac{
	\D(\diag(y)^{\langle i\rangle},I^{\langle i\rangle},
        M_2^{\langle i\rangle},\ldots,M_{n-2}^{\langle i\rangle})
	}{
	\D(I^{\langle i\rangle},I^{\langle i\rangle},
        M_2^{\langle i\rangle},\ldots,M_{n-2}^{\langle i\rangle})
	},\\
	p_i &:= \frac{1}{n}
	\D(I^{\langle i\rangle},I^{\langle i\rangle},
        M_2^{\langle i\rangle},\ldots,M_{n-2}^{\langle i\rangle})
\end{align*}
for $y\in\mathbb{R}^n$.
Then $Q(x,y) = \langle x,Ay\rangle_{\ell^2(p)}$, where
$\langle x,y\rangle_{\ell^2(p)}:=\sum_i x_iy_ip_i$. As $Q(x,y)$ is
symmetric, $A$ is self-adjoint on $\ell^2(p)$. Moreover, clearly 
$A1=1$. Finally, note that $A$ is a positive matrix by Lemma 
\ref{lem:mdprop}(\textit{e}). Therefore, by the Perron-Frobenius
theorem \cite[Theorem 1.4.4]{BR97}, $A$ has largest eigenvalue $1$ and
this eigenvalue is simple.

Now recall that we assumed the validity of Theorem \ref{thm:a} for 
$(n-1)$-dimensional matrices. The latter implies, exactly as in
the proof of Lemma \ref{lem:lich}, that
$$
	(Ay)_i^2p_i \ge 
	\frac{1}{n}
	\D(\diag(y)^{\langle i\rangle},\diag(y)^{\langle i\rangle},
        M_2^{\langle i\rangle},\ldots,M_{n-2}^{\langle i\rangle}).
$$
Summing both sides over $i$ and applying Lemma 
\ref{lem:mdprop}(\textit{f}) yields
$$
	\langle Ay,Ay\rangle_{\ell^2(p)}
	\ge
	\D(I,\diag(y),\diag(y),M_2,\ldots,M_{n-2})
	=
	\langle y,Ay\rangle_{\ell^2(p)}.
$$
By choosing $y$ to be an eigenvector of $A$, we find that any eigenvalue 
$\lambda$ of $A$ 
satisfies $\lambda^2\ge\lambda$, so $\lambda\ge 1$ or $\lambda\le 0$.
But as $1$ is the maximal eigenvalue and this eigenvalue is simple,
we have shown that $A$ has a one-dimensional positive eigenspace.
Therefore, Lemma \ref{lem:hyperfull}($3\Rightarrow 1$) implies the desired 
conclusion $Q(x,1)^2\ge Q(x,x)\,Q(1,1)$.
\end{proof}

It remains to show that the mixed discriminant inequality for arbitrary
$n$-dimensional matrices can be reduced to the special case of Lemma 
\ref{lem:mdcommut}.

\begin{cor}
Let $n\ge 3$ and let $B,M_1,\ldots,M_{n-2}$ be $n$-dimensional positive 
semidefinite matrices.
Then for any $n$-dimensional symmetric matrix $A$, we have
$$
	\D(A,B,M_1,\ldots,M_{n-2})^2 \ge
	\D(A,A,M_1,\ldots,M_{n-2})\,\D(B,B,M_1,\ldots,M_{n-2}).
$$
\end{cor}

\begin{proof}
We may assume without loss of generality that $M_1,\ldots,M_{n-2}$ are 
positive definite (otherwise, replace $M_i$ by $M_i+\varepsilon I$ and let 
$\varepsilon\to 0$ at the end). Moreover, applying Lemma 
\ref{lem:mdprop}(\textit{c}) with $U=M_1^{-1/2}$, we may assume that 
$M_1=I$.

We now define the quadratic form $\mathbf{Q}(Z,Z'):= 
\D(Z,Z',I,M_2,\ldots,M_{n-2})$ 
on the space of $n$-dimensional symmetric matrices. It follows immediately 
from Lemma \ref{lem:mdcommut} and Lemma \ref{lem:mdprop}(\textit{e}) 
that $\mathbf{Q}(Z,I)=0$ implies $\mathbf{Q}(Z,Z)\le 0$
for any diagonal matrix $Z$. The same conclusion follows for any 
symmetric matrix $Z$, as we may always reduce to the diagonal case
by a change of basis. Thus
$\mathbf{Q}(A,B)^2\ge \mathbf{Q}(A,A)\,\mathbf{Q}(B,B)$
by Lemma \ref{lem:hyperfull}($2\Rightarrow 1$), which
concludes the proof.
\end{proof}

\section{An alternative approach using polytopes}
\label{sec:app}

Two different approaches to the proof of the Alexandrov-Fenchel inequality 
appear already in Alexandrov's work. One approach \cite{Ale38} focuses 
attention on smooth bodies, which gives rise to elliptic operators. The 
other (historically earlier) approach \cite{Ale37} is to focus instead on 
polytopes. Because polytopes have a finite number of normal directions, 
the role of elliptic operators is replaced here by finite-dimensional 
matrices. The latter may be considered more ``elementary'', in that the 
proof requires in principle only linear algebra and basic geometry.

The present authors find computations with polytopes somewhat less clean 
and intuitive than the smooth approach. Nonetheless, the polytope method 
is of interest in its own right. The aim of this section is to sketch how 
our methods may be implemented in the polytope setting. The following 
discussion is not fully self-contained; we refer to \cite{Sch14} for 
proofs of the basic polytope representations of mixed volumes, and focus 
on adapting the our methods to this context.

Let $P_1,\ldots,P_n$ be polytopes in $\mathbb{R}^n$. We denote by
$F(P,u)$ the face of the polytope $P$ with normal direction $u\in S^{n-1}$.
The following expression\footnote{
	By definition $F(P_i,u)$, $i=2,\ldots,n$
	all lie in the $(n-1)$-dimensional space $u^\perp
	\subset\mathbb{R}^n$ modulo translation. By a slight abuse
	of notation, we denote by $\V(F(P_2,u),\ldots,F(P_n,u))$ the
	$(n-1)$-dimensional mixed volume of the translated faces in
	$u^\perp$ (cf.\ Remark \ref{rem:mdlowdim}).
}
is the analogue for polytopes of
the representation \eqref{eq:mvrep} of mixed volumes of $C^2_+$ bodies
\cite[(5.23)]{Sch14}:
\begin{equation}
\label{eq:polymvrep}
	\V(P_1,\ldots,P_n) = 
	\frac{1}{n}\sum_{u\in S^{n-1}}
	h_{P_1}(u)\,\V(F(P_2,u),\ldots,F(P_n,u)).
\end{equation}
Implicit in the notation is that 
$\V(F(P_2,u),\ldots,F(P_n,u))$ is nonzero only at a finite number of
points $u$ on the sphere; it suffices to restrict the sum to 
the normal directions of the facets ($(n-1)$-dimensional faces) of 
$P_2+\cdots+P_n$.

We would like to think of the restriction of $h_{P_i}$ to the relevant 
normal directions as finite-dimensional vectors, and of mixed volume as a 
quadratic form of such vectors. The problem with \eqref{eq:polymvrep} is 
that $\V(F(P_2,u),\ldots,F(P_n,u))$ is not naturally expressed in terms of 
$h_{P_2}$, but rather in terms of $h_{F(P_2,u)}$. It is therefore unclear 
how we may view \eqref{eq:polymvrep} as a quadratic form of the support 
vectors of the original polytopes. It turns out that this can be done, and 
that one can recover various properties of mixed volumes that appeared 
naturally in the smooth setting, if one restricts attention to certain 
``nice'' families of polytopes.

In the following, we will call polytopes $P_1,\ldots,P_n$ \emph{strongly 
isomorphic} if
$$
	\dim F(P_1,u)=\dim F(P_2,u) =\cdots=\dim F(P_n,u)
	\mbox{ for all }u\in S^{n-1}.
$$
In this setting, the sum in \eqref{eq:polymvrep} ranges over the common 
normal directions $\Omega$ of the facets of $P_i$, and $h_{F(P_i,u)}$ is 
a linear function (independent of $i$) of the restriction of $h_{P_i}$ to 
$\Omega$ \cite[p.\ 276]{Sch14}. We 
also recall that a polytope $P$ in $\mathbb{R}^n$ is called \emph{simple} 
if it has nonempty interior and each vertex is contained in exactly $n$ 
facets.

\begin{lem}
\label{lem:iso}
Let $P_3,\ldots,P_n$ be simple strongly isomorphic polytopes in 
$\mathbb{R}^n$, 
and let $\Omega\subset S^{n-1}$ be the common normal directions of facets 
of $P_i$. Denote by $\uh_{P_i}:=(h_{P_i}(u))_{u\in\Omega} 
\in\mathbb{R}^{|\Omega|}$ the support vector of $P_i$. Then:
\begin{enumerate}[(a)]
\itemsep\abovedisplayskip
\item For every $x\in\mathbb{R}^{|\Omega|}$ and polytope $P$ 
strongly isomorphic to $P_i$, there is a polytope $Q$ strongly 
isomorphic to $P_i$ and $a>0$ such that $x=a(\uh_Q-\uh_P)$.
\item There is a $|\Omega|$-dimensional symmetric matrix $\tilde A$ such that
$$
	(\tilde A \uh_P)_u=\frac{1}{n}\V(F(P,u),F(P_3,u),\ldots,F(P_n,u))
$$
for every $u\in\Omega$ and polytope $P$ strongly isomorphic to $P_i$.
\item $\tilde A = L+D$ for an irreducible 
nonnegative matrix $L$ and diagonal matrix $D$.
\end{enumerate}
Moreover, any family of convex bodies $C_1,\ldots,C_n$ can be approximated 
arbitrarily well in the Hausdorff metric by simple strongly isomorphic 
polytopes $P_1,\ldots,P_n$.
\end{lem}

\begin{proof}
Part (\textit{a}) follows from \cite[Lemma 
2.4.13]{Sch14}. Parts (\textit{b}) and (\textit{c}) may be read off from the 
explicit expression given in the proof of \cite[Lemma 5.1.5]{Sch14}; in 
particular, irreducibility follows as the facet graph of a polytope is 
connected (this standard fact follows by duality from \cite[Theorem 
15.5]{Bro83}).
That arbitrary bodies may be approximated by simple strongly 
isomorphic polytopes is \cite[Theorem 2.4.15]{Sch14}.
\end{proof}

In comparison with the smooth setting, part (\textit{a}) of this lemma is 
analogous to Corollary \ref{cor:c2}; $\tilde A$ is analogous to 
\eqref{eq:atilde}; and part (\textit{c}) corresponds to ellipticity.

It will be convenient to
extend mixed volumes linearly as follows: whenever $x=\uh_{Q}-\uh_{Q'}$
for polytopes $Q,Q'$ strongly isomorphic to $P_i$, we define 
$$
	\V(x,P_2,\ldots,P_n):=\V(Q,P_2,\ldots,P_n)-\V(Q',P_2,\ldots,P_n),
$$
and for $u\in\Omega$
\begin{align*}
	&\V(F(x,u),F(P_3,u),\ldots,F(P_n,u)) := \mbox{}
	\\& \qquad \V(F(Q,u),F(P_3,u),\ldots,F(P_n,u)) 
	-\V(F(Q',u),F(P_3,u),\ldots,F(P_n,u))
\end{align*}
(the latter notation is justified by Lemma \ref{lem:iso}(\textit{b})).
By Lemma \ref{lem:iso} and the representation \eqref{eq:polymvrep}, we can 
then write for any $x,y\in\mathbb{R}^{|\Omega|}$
\begin{align*}
	(\tilde Ax)_u &=
	\frac{1}{n}\V(F(x,u),F(P_3,u),\ldots,F(P_n,u)),
	\\
	\langle x,\tilde Ay\rangle &=
	\V(x,y,P_3,\ldots,P_n).
\end{align*}
We are now ready to prove the Alexandrov-Fenchel inequality for polytopes.

\begin{thm}
\label{thm:polyaf}
Let $P,P_3,\ldots,P_n$ be simple strongly isomorphic polytopes in 
$\mathbb{R}^n$ with common facet directions $\Omega\subset S^{n-1}$. Then 
for every $x\in\mathbb{R}^{|\Omega|}$
$$
	\V(x,P,P_3,\ldots,P_n)^2 \ge
	\V(x,x,P_3,\ldots,P_n)\V(P,P,P_3,\ldots,P_n).
$$
In particular, by the last part of Lemma \ref{lem:iso},
this implies Theorem \ref{thm:af}.
\end{thm}

\begin{proof}
The proof will proceed by induction on the dimension $n$.

For $n=2$, the Alexandrov-Fenchel inequality 
$\V(K,L)^2\ge\V(K,K)\,\V(L,L)$ follows easily from the Brunn-Minkowski 
theorem \cite[Theorem 7.2.1]{Sch14}. This implies the result when 
$x=\uh_Q$ is the support vector of a polytope strongly isomorphic to $P$.
The general case $x\in\mathbb{R}^{|\Omega|}$ now follows from 
Lemma \ref{lem:iso}(\textit{a}) as in Remark \ref{rem:generalf}.

We now proceed to the induction step; that is, we will assume the 
theorem is valid for polytopes in $\mathbb{R}^{n-1}$ with $n\ge 3$, and 
aim to conclude it is also valid for polytopes in $\mathbb{R}^n$. 
To this end, define the $|\Omega|$-dimensional matrix $A$ and
$p\in\mathbb{R}^{|\Omega|}$ as
\begin{align*}
	(Ax)_u &:= 
	\frac{h_{P_3}(u) \V(F(x,u),F(P_3,u),\ldots,F(P_n,u))}
	{\V(F(P_3,u),F(P_3,u),\ldots,F(P_n,u))}, \\
	p_u &:= 
	\frac{1}{n}\frac{\V(F(P_3,u),F(P_3,u),\ldots,F(P_n,u))}{h_{P_3}(u)}
\end{align*}
(as in section \ref{sec:3}, we assume without loss of generality
that $h_{P_3}>0$). By definition, $\V(x,y,P_3,\ldots,P_n)=\langle 
x,Ay\rangle_{\ell^2(p)}$. Thus, as mixed volumes are symmetric, $A$ 
is self-adjoint on $\ell^2(p)$. Moreover, $A$ was defined so that 
$A\uh_{P_3}=\uh_{P_3}$. By Lemma \ref{lem:iso}(\textit{c}), the 
Perron-Frobenius theorem \cite[Theorem 1.4.4]{BR97} (applied to $A+cI$ for 
$c$ sufficiently large) implies $A$ has largest eigenvalue $1$ and that
this is a simple eigenvalue.

Now note that the facets of simple strongly isomorphic polytopes 
with a given normal direction are simple (cf.\ \cite[Theorem 12.15]{Bro83} 
for this basic fact) and strongly isomorphic (by definition). 
Thus the induction hypothesis implies
\begin{align*}
	(Ax)_u^2p_u 
	&= 
	\frac{h_{P_3}(u)}{n}
	\frac{\V(F(x,u),F(P_3,u),\ldots,F(P_n,u))^2}
	{\V(F(P_3,u),F(P_3,u),\ldots,F(P_n,u))}	
	\\
	&\ge
	\frac{h_{P_3}(u)}{n}
	\V(F(x,u),F(x,u),F(P_4,u),\ldots,F(P_n,u)).
\end{align*}
Summing over $u$ and using \eqref{eq:polymvrep}
and symmetry of mixed volumes yields
$$
	\langle Ax,Ax\rangle_{\ell^2(p)}
	\ge
	\V(P_3,x,x,P_4,\ldots,P_n)
	=
	\langle x,Ax\rangle_{\ell^2(p)}.
$$
Choosing $x$ to be an eigenvector of $A$, we find that any eigenvalue
$\lambda$ of $A$ satisfies $\lambda^2\ge\lambda$, so $\lambda\ge 1$ or
$\lambda\le 0$. But as $1$ is the maximal eigenvalue of $A$ and as it is 
a simple eigenvalue,
the conclusion follows immediately from Lemma \ref{lem:hyper}.
\end{proof}

\section{Concluding remarks}
\label{sec:concl}

\subsection{Alexandrov's proof}
\label{sec:alex}

Alexandrov's proof of the Alexandrov-Fenchel inequality \cite{Ale38} is 
very different in spirit than the method used in section \ref{sec:3}. For 
sake of comparison, let us briefly sketch his approach.

Despite the evident similarity between Theorems \ref{thm:af} and 
\ref{thm:a}, the mixed discriminant inequality is not used in a direct 
manner in Alexandrov's proof. Rather, it is used to establish an 
apparently unrelated fact: that the kernel of $\mathscr{A}$ has dimension 
$n$ (it consists precisely of first-order spherical harmonics). Once this 
is known, one may establish the requisite spectral property of 
$\mathscr{A}$ by a homotopy method. For a special choice of bodies (e.g., 
as in section \ref{sec:bochner} below), an explicit computation shows that 
the positive eigenspace is one-dimensional. We now interpolate between 
these special bodies and the given bodies in Theorem \ref{thm:af}. If the 
dimension of the positive eigenspace were to increase, then an eigenvalue 
must cross from below zero to above zero. But then the kernel of the 
operator must have dimension larger than $n$ at the crossing point, which 
yields a contradiction.

In contrast, our method appears conceptually and technically simpler, as 
the mixed discriminant inequality yields the Alexandrov-Fenchel 
inequality directly by a one-line computation. In particular, we have no 
need to characterize any other properties of the operator in the proof 
(such as its kernel). Let us also note that our normalization of 
$\mathscr{A}$ is slightly different than the one employed by Alexandrov: 
Alexandrov defined the operator so that $h_L$, rather than $h_{C_1}$, is 
its top eigenvector. With this special choice, the final inequality 
follows directly without appealing to Lemma \ref{lem:hyper}. However, in 
our approach, the choice $h_{C_1}$ (or, equivalently, $h_{C_i}$ for some 
$i$) plays a special role in the proof of Lemma \ref{lem:lich}. By fully 
exploiting Lemma~\ref{lem:hyper} we gain significant flexibility, as is 
further illustrated in section \ref{sec:4}.

\subsection{Equality cases}

It is not hard to deduce from the proof of Lemma~\ref{lem:hyperfull} that 
equality $\langle x,Ay\rangle^2=\langle x,Ax\rangle\langle y,Ay\rangle$ 
holds when $\langle y,Ay\rangle>0$ if and only if 
$x-ay\in\mathop{\mathrm{ker}}A$ for some $a\in\mathbb{R}$. Thus 
Alexandrov's proof (cf.\ section \ref{sec:alex}), while somewhat 
circuitous, does provide additional information: it shows that equality 
holds in Theorem \ref{thm:af} for \emph{smooth} bodies if and only if 
$h_K-ah_L$ is a linear function, i.e., when $K$ and $L$ are homothetic. 
(This is false for nonsmooth bodies, for which the characterization of 
equality cases remains open; cf.\ \cite[section 7.6]{Sch14}.)

Let us briefly sketch how the equality cases can be deduced from our 
approach. Let $f\in\mathop{\mathrm{ker}}\mathscr{A}$. Then the inequality
in Lemma \ref{lem:lich} holds with equality, and thus all inequalities
in its proof must hold with equality. In particular, one has
equality in Theorem \ref{thm:a} with $A=D^2f$, $B=D^2h_{C_1}$, and
$M_i=D^2h_{C_{i+1}}$. It is known that equality holds in Theorem 
\ref{thm:a} when $B,M_i>0$ if and only if $A=\lambda B$ for some 
$\lambda\in\mathbb{R}$. Thus $D^2f-\lambda D^2h_{C_1}=0$ for some 
$\lambda:S^{n-1}\to\mathbb{R}$. But as $\mathscr{A}f=0$, we have
$$
	0 = \frac{\D(D^2f-\lambda D^2h_{C_1},D^2h_{C_1},\ldots,
	D^2h_{C_{n-2}})}{\D(D^2h_{C_1},D^2h_{C_1},\ldots,
        D^2h_{C_{n-2}})}=
	-\lambda.
$$
Thus we have shown that $D^2f=0$, so
$f$ must be a linear function.

Using similar reasoning, the abovementioned equality 
cases of Theorem \ref{thm:a} may be deduced from the proof given in 
section \ref{sec:4}. We can similarly recover the equality cases of 
Theorem \ref{thm:polyaf}. We omit the details in the interest of space.

\subsection{The Bochner method}
\label{sec:bochner}

The simple technique of this paper has its origin in the classical bound 
of Lichnerowicz on the spectral gap of the Laplacian on Riemannian 
manifolds with positive Ricci curvature \cite{Lic58}. This connection goes 
beyond an analogy between the proofs, as we will presently explain.

Let us briefly recall Lichnerowicz' argument. Let $M$ be an
$(n-1)$-dimensional compact Riemannian manifold. We denote by $\nabla_M$ 
the covariant derivative and by $\Delta_M$ the Laplacian. The basic 
observation of Lichnerowicz is that, by integrating the classical Bochner 
formula, one obtains the identity
(cf.\ \cite[Theorem 4.70]{GHL04})
\begin{align}
\nonumber
	\int_M (\Delta_M f)^2 &=
	\frac{n-1}{n-2}
	\int_M \mathrm{Ric}_M(\nabla_Mf,\nabla_Mf) \\ &\quad +
	\frac{1}{n-2}
	\int_M \bigg\{(n-1)\mathrm{Tr}[(\nabla_M^2f)^2]-
	\mathrm{Tr}[\nabla_M^2f]^2\bigg\}.
\label{eq:boch}
\end{align}
Note that the last term in this expression is always nonnegative by
Cauchy-Schwarz. If we specialize to the sphere $M=S^{n-1}$, the Ricci 
curvature tensor is given by $\mathrm{Ric}_{S^{n-1}}(X,X)=(n-2)\|X\|^2$,
and we obtain after integrating by parts
\begin{equation}
\label{eq:sqlapl}
	\int_{S^{n-1}} (\Delta_{S^{n-1}}f)^2\,d\omega
	\ge -(n-1)\int_{S^{n-1}} f\Delta_{S^{n-1}}f\,d\omega.
\end{equation}
Thus every eigenvalue $\lambda$ of $-\Delta_{S^{n-1}}$ (which is positive
semidefinite) must satisfy $\lambda^2\ge (n-1)\lambda$, that is,
$\lambda=0$ or $\lambda\ge n-1$. As noted by Lichnerowicz, this
argument applies to any Riemannian manifold
$M$ with $\mathrm{Ric}_M(X,X)\ge (n-2)\|X\|^2$.

The idea of Lichnerowicz to use an identity for $(\Delta_M f)^2$ to deduce 
spectral estimates for $\Delta_M$ forms the foundation for our proof of 
the Alexandrov-Fenchel inequality. However, the proof of 
\eqref{eq:sqlapl}, using the Bochner formula, is very different than the 
proof of Lemma \ref{lem:lich}. Remarkably, it turns out that not only the 
inequality \eqref{eq:sqlapl}, but even the Bochner identity 
\eqref{eq:boch} for $M=S^{n-1}$, is implicit in the proof of Lemma 
\ref{lem:lich}. Thus we may truly think of our method as a ``Bochner 
method''.

To recover \eqref{eq:boch} for $M=S^{n-1}$ from the proof of Lemma
\ref{lem:lich}, we consider the special case where 
$C_1=\cdots=C_{n-2}=B_2$ is the Euclidean ball. Then
$h_{B_2}=1$ and $D^2h_{B_2}=I$. Differentiating $\det(I+tA)$ 
with respect to $t$ and using \eqref{eq:defmd} yields
\begin{align*}
	&\D(I,\ldots,I) = \det(I) = 1,
	\\&
	\D(A,I,\ldots,I) = \frac{1}{n-1}\mathrm{Tr}[A], \\
	&\D(A,A,I,\ldots,I) = \frac{1}{(n-1)(n-2)}(
	\mathrm{Tr}[A]^2-\mathrm{Tr}[A^2]).
\end{align*}
Moreover, by differentiating the 1-homogeneous extension $\|x\| 
f(x/\|x\|)$ of $f$, we find that $D^2f=\nabla_{S^{n-1}}^2f+fI$ in terms of 
the covariant Hessian. In particular, we obtain in this special case
$\mathscr{A}f = \frac{1}{n-1}\Delta_{S^{n-1}}f+f$,
$d\mu = \frac{1}{n}d\omega$. We now
compute
\begingroup
\allowdisplaybreaks
\begin{align*}
	&\int (\Delta_{S^{n-1}}f)^2\,d\omega 
	+ (n-1)\int f\Delta_{S^{n-1}}f\,d\omega
	\\ &\quad=
	(n-1)^2\bigg(
	\int (\mathscr{A}f)^2\,d\omega 
	- \int f\mathscr{A}f\,d\omega\bigg)
	\\
	&\quad=
	(n-1)^2
	\int \{\D(D^2f,I,\ldots,I)^2-\D(D^2f,D^2f,I,\ldots,I)\}\,d\omega
	\\
	&\quad=
	\frac{1}{n-2}
	\int
	\{(n-1)\mathrm{Tr}[(\nabla_{S^{n-1}}^2f)^2]
	- \mathrm{Tr}[\nabla_{S^{n-1}}^2f]^2
	\}\,d\omega.
\end{align*}
\endgroup
Here the first equality follows by completing the square; the second
equality is a reformulation of the proof of Lemma \ref{lem:lich}; 
and the third equality uses the explicit expressions for mixed volumes and 
$D^2f$ given above. Thus we recovered \eqref{eq:boch} for $M=S^{n-1}$
as a special case of the proof of Lemma \ref{lem:lich}.

The connections hinted at here can be developed in far greater generality;
however, as the geometric approach is somewhat tangential to the theme of 
this paper, we omit further discussion. Related ideas, inspired by
complex geometry, were also obtained by D.\ Cordero-Erausquin and B.\ 
Klartag (personal communication).

\subsection*{Acknowledgment}

We are grateful to Joel Tropp for helpful comments and for pointing out 
some inaccuracies in an earlier version of this paper, and to a referee 
for detailed comments that have helped us significantly improve the 
presentation.

\bibliographystyle{abbrv}
\bibliography{ref}

\end{document}